\documentclass[12pt]{amsart}
\usepackage{epsfig,color}
\usepackage{blindtext}
\usepackage{hyperref}
\usepackage{graphicx}
\usepackage{enumitem} 
\usepackage{url}
\usepackage{amssymb}
\usepackage{nicefrac,mathtools}
\usepackage{comment}
\usepackage{cancel} 
\usepackage[normalem]{ulem} 
\usepackage{epstopdf}

\usepackage{cancel} 
\usepackage[normalem]{ulem} 
\theoremstyle{definition}

\newtheorem{theorem}{Theorem}[section]
\newtheorem{definition}[theorem]{Definition}

\newtheorem{lemma}[theorem]{Lemma}
\newtheorem{remark}[theorem]{Remark}
\newtheorem{corollary}[theorem]{Corollary}

\newtheorem*{acknowledgements}{Acknowledgements}
\newtheorem{claim}[]{Claim}
\newtheorem*{remark*}{Remark}

\numberwithin{equation}{section}

\newcommand{\lb}[1]{\langle#1\rangle}

\newcommand{\mf}{\mathbf}
\newcommand{\mc}{\mathcal}
\newcommand{\mb}{\mathbb}

\newcommand{\wti}{\widetilde}

\newcommand{\Si}{\Sigma}
\newcommand{\la}{\lambda}
\newcommand{\R}{\mathbb{R}}

\newcommand{\dist}{\operatorname{dist}}

\newcommand{\interior}{\operatorname{int}}

\DeclareMathOperator{\area}{Area}
\DeclareMathOperator{\Index}{index}

\DeclareMathOperator{\Ric}{Ric}

\DeclareMathOperator{\Clos}{Clos}
\DeclareMathOperator{\spt}{spt}
\DeclareMathOperator{\Graph}{Graph}

\title[Compactness and generic finiteness]{Compactness and generic finiteness for free boundary minimal hypersurfaces (I)}

\author{Qiang Guang}
\address{Department of Mathematics, University of California Santa Barbara, Santa Barbara, CA 93106, USA}
\email{guang@math.ucsb.edu}

\author{Zhichao Wang}
\address{Max-Planck Institute for Mathematics, Vivatsgasse 7, 
53111 Bonn, Germany}
\email{wangzhichaonk@gmail.com}

\author{Xin Zhou}
\address{Department of Mathematics, University of California Santa Barbara, Santa Barbara, CA 93106, USA}
\email{zhou@math.ucsb.edu}

\subjclass[2010]{Primary 53A10, 53C42}

\keywords{free boundary minimal surfaces, compactness, Jacobi fields, curvature estimates}

\begin{document}

\begin{abstract}
Given a  compact Riemannian manifold with boundary, we prove that the space of embedded, which may be improper, free boundary minimal hypersurfaces with uniform area and Morse index upper bound is compact in the sense of smoothly graphical convergence away from finitely many points.  We show that the limit of a sequence of such hypersurfaces always inherits a non-trivial Jacobi field when it has multiplicity one. In a forthcoming paper, we will construct Jacobi fields when the convergence has higher multiplicity.
\end{abstract}

\maketitle

\section{Introduction}

Let $(M^{n+1},\partial M)$ be a compact Riemannian manifold with smooth boundary. A smooth embedded hypersurface $\Sigma^n\subset M^{n+1}$ is said to be a {\em free boundary minimal hypersurface}, if $\Sigma$ has vanishing mean curvature and $\partial \Sigma$ meets $\partial M$ orthogonally. For simplicity, we use \emph{FBMH} to denote free boundary minimal hypersurface. FBMHs arise variationally as critical points of the area functional among all hypersurfaces in $M$ with boundary constrained freely on $\partial M$. The mathematical investigation of FBMHs dates back at least to Courant \cite{Cou40} and Lewy \cite{Le51}, and there were intense study of this subject afterward, e.g. \cite{HN79, MY82b, Str84, GJ86, Jost86, Ye91, Fra00}. Many new progresses, especially on the existence theory of FBMHs, were made in recent years. Among them, Schoen-Fraser \cite{FrSc1, FrSc2} constructed many examples of free boundary minimal surfaces in the round three-ball and found a deep relation of them with the extremal eigenvalue problem. More examples of FBMHs in the round three-ball were recently found by Folha-Pacard-Zolotareva \cite{FPZ17}, Ketover \cite{Ket16b}, Kapouleas-Li \cite{KL17} and Kapouleas-Wiygul \cite{KW17}. Maximo-Nunes-Smith \cite{MNS17} constructed an annuli type FBMH in certain convex three-manfolds using degree theory.  Last but foremost, to produce FBMHs in an arbitrary compact manifold with boundary, Almgren \cite{Alm62, Alm65} in 1970s initiated a program on establishing a global variational theory for FBMHs via the min-max method. This program was finished by the last author with M. Li \cite{LZ16} recently; we also refer to \cite{GJ86, Jost86, Li15, DeRa16} for partial results where certain topological and boundary convexity assumptions were made. FBMHs produced by the min-max theory in \cite{Alm62, Alm65, LZ16} are usually called {\em min-max FBMHs}. One key novelty in \cite{LZ16} is that the min-max FBMHs are allowed to be improper, or equivalently, the interior of the min-max FBMH may touch the boundary of the ambient space. It is also conjectured in \cite{Alm65} that the Morse index of min-max FBMHs are bounded from above by the number of parameters in the min-max construction.

The purpose of this and a followup papers is to establish compactness and generic finiteness results for FBMHs satisfying uniform area and Morse index upper bounds. Applications of our results will include the proof of the Morse index upper bound conjectured by Almgren \cite{Alm65}; see \cite{GWZ18}. Given a FBMH $\Sigma\subset M$, the proper subset $\mc{R}(\Sigma)\subset \Sigma$ is the complement of the touching set $\mc{S}(\Sigma)=\mathrm{int}(\Sigma)\cap\partial M$, i.e., $\mc{R}(\Sigma)=\Sigma\setminus\left(\mathrm{int}(\Sigma)\cap\partial M\right)$. For any $I\in \mb{N}$ and $C>0$, we use $\mc{M}(I,C)$ to denote the set of almost properly embedded FBMHs $\Sigma$ in $M$ satisfying $\area(\Sigma)\leq C$ and the Morse index of $\Sigma$ on the proper subset $\mc{R}(\Sigma)$ is bounded by $I$.

Our first main result is the following compactness theorem for FBMHs.

\begin{theorem}
\label{thm:main compactness}
Assume that $2\leq n\leq 6$. Let $M^{n+1}$ be a compact Riemannian manifold with boundary $\partial M$. For fixed $I\in \mathbb{N}$ and $C_0>0$, suppose that 
$\{\Sigma_k\}$ is a sequence of almost properly embedded FBMHs in $M$ with $\area(\Sigma_k)\leq C_0$ and
\begin{equation}\label{Morse index upper bound} \text{the Morse index of $\Sigma_k$ on the proper subset $\mc{R}(\Sigma_k)$ is bounded by $I$}.
\end{equation}
Then there exists a finite set of points $\mathcal{W}\subset M$ with $\#(\mathcal{W})\leq I$ and an almost properly embedded FBMH $\Sigma_\infty \subset M$  such that, after passing to a subsequence, $\Sigma_k$ converges smoothly and locally uniformly to $\Sigma_\infty$ on $\Sigma_\infty\setminus \mc{W}$ with finite multiplicity. Furthermore, the Morse index of $\Sigma_\infty$ on the proper subset $\mc{R}(\Sigma_\infty)$ is bounded by $I$ and $\area(\Sigma_\infty)\leq C_0$.

\end{theorem}

Our second main result is the construction of a non-trivial Jacobi field on the limit hypersurface 
when the convergence in Theorem \ref{thm:main compactness} has multiplicity one. The higher multiplicity cases will be settled in the forthcoming paper.

\begin{theorem}\label{thm:Jacobi:1}
Assume that $2\leq n\leq 6$. Let $M^{n+1}$ be a compact Riemannian manifold with boundary $\partial M$. For fixed $I\in \mb{N}$ and $C_0>0$, suppose that $\{\Sigma_k\}$ is a sequence with $\area(\Sigma_k)\leq C_0$ and satisfying (\ref{Morse index upper bound}) in $\mc{M}(I,C_0)$, and that $\Sigma_k$ converges smoothly and locally uniformly to some limit $\Sigma\in \mc{M}(I,C_0)$ on $\Sigma\setminus \mc{W}$ with multiplicity one, where $\mc{W}\subset M$ is a finite set of points with $\#(\mc{W})\leq I$.  
 Assume that $\Sigma_k\neq \Sigma$ eventually. Then
\begin{enumerate}
\item If $\Sigma$ is two-sided, then $\Sigma$ has a non-trivial Jacobi field.

\item If $\Sigma$ is one-sided, then $\wti{\Sigma}$ has a non-trivial Jacobi field, where $\wti{\Sigma}$ is the double cover of $\Sigma$.

\end{enumerate}
\end{theorem}

\begin{remark}
We remark that 
Ambrozio-Carlotto-Sharp \cite{ACS17} first investigated similar compactness results where all FBMHs under their consideration are proper, or equivalently $\mc{S}(\Sigma)=\emptyset$ (see more discussions later). Compared with \cite{ACS17}, the key novelty of our results lies in the following two aspects:
     \begin{enumerate}[label=(\roman*)]
     \item\label{first challenge} We prove a new curvature estimate for FBMHs which are only stable away from the touching set;
     \item\label{second challenge} For a sequence of FBMHs that converges in the above sense to a limit, if a sequence of their boundary components collapses to a point in the limit, we design a new scheme to construct a non-trivial Jacobi field on the limit hypersurface.
     \end{enumerate}
\end{remark}

As a direct consequence of Theorem \ref{thm:Jacobi:1} and the main theorem in the followup paper, we obtain the generic finiteness theorem for FBMHs. That is,  if $M^{n+1}$ is a compact manifold with boundary and $2\leq n\leq 6$, then for a generic metric on $M$ and fixed $I\in \mb{N}$ and $C_0>0$ , there are only finitely many  almost properly embedded FBMHs in $M$ satisfying $\area(\Sigma)\leq C_0$ and (\ref{Morse index upper bound}).

Now we provide a brief history of compactness results for minimal surfaces, and we start with closed minimal hypersurfaces in closed manifolds. Choi-Schoen \cite{CS85} proved compactness for minimal surfaces with bounded topology in three-manifolds with positive Ricci curvature, and their result was later improved by Anderson \cite{An85} and White \cite{Wh87} under area and topology bound assumptions. Without assuming area upper bound, Colding-Minicozzi \cite{CM10, CM11, CM12, CM13, CM14} developed a whole theory of lamination convergence for minimal surfaces with bounded topology in three-manifolds. In higher dimensions, Schoen-Simon-Yau \cite{SSY75} and Schoen-Simon \cite{SS81} proved interior curvature estimates and compactness for stable minimal hypersurfaces with uniform area upper bound. Their results were recently generalized by Sharp \cite{Sharp17} to minimal hypersurfaces with uniform Morse index and area upper bound. Without area upper bound assumption, the last author and H. Li \cite{LiZh16} obtained lamination convergence for minimal surfaces with uniform Morse index bound in three-manifolds. We also refer to \cite{ACS16, Chodosh-Ketover-Maximo15, Carlotto15} for recent development along this direction.

Concerning FBMHs, Fraser-Li \cite{FrLi} obtained the first compactness result for FBMHs in three-manifolds with non-negative Ricci curvature and convex boundary, as a natural free boundary analog of Choi-Schoen's result \cite{CS85}.  In higher dimensions, Guang, Zhou and M. Li \cite{GLZ16} proved curvature estimates and compactness for globally stable FBMHs with uniform area upper bound, as natural analogs of Schoen-Simon-Yau \cite{SSY75} and Schoen-Simon's results \cite{SS81}; the results in \cite{GLZ16} played an essential role in the free boundary min-max theory \cite{LZ16}. Very recently, Ambrozio-Carlotto-Sharp \cite{ACS17} proved compactness for FBMHs which are properly embedded and have uniform Morse index and area upper bounds. The novelty of \cite{ACS17} includes a boundary removable singularity result, a scheme to construct a Jacobi field when a sequence of FBMHs converges to a limit and none of the boundary components degenerate to a point, and a bumpy metric theorem generalizing those of White \cite{Whi91, Whi15}. Even though, there are essential new challenges \ref{first challenge}, \ref{second challenge} when the FBMHs are allowed to be improper. We overcome these difficulties by several new ideas, which we believe will be useful in other problems related to FBMHs.

\vspace{1em}
Now we present an overview of our paper, with emphasize on our new ideas. The first main ingredient is a new curvature estimate for FBMHs which are only stable away from the touching set. The curvature estimates in \cite{GLZ16} require a FBMH to be globally stable, even across the touching set, whilst this new estimate only need the FBMH to be stable away from the touching set. To state the result, a few notions are made as follows. For any subset $A\subset M$, we use $\partial_{rel} A$ to denote the \emph{relative boundary} of $A$, that is, the set of boundary points of $A$ which are in the interior of $M$. We can also assume that $M$ is a compact domain of a closed Riemannian manifold $\wti{M}^{n+1}$ with the same dimension. The precise statement of the curvature estimate is the following, which may be of independent interest.

\begin{theorem}\label{thm:main:interior}
Let $(M^{n+1},g)$ be a compact Riemannian manifold with boundary $\partial M$ and $2\leq n\leq 6$. Let $U\subset \subset M$ be a relative open subset. Suppose $\Sigma^n\subset U$ is a smooth compact embedded minimal hypersurface in $M$ with free boundary lying on $\partial M\cap U$ and $\area(\Sigma)\leq C_0$. If $\Sigma$ is stable away from the touching set $\mc{S}(\Sigma)=\text{int}(\Sigma)\cap \partial M$, then
\[|A|^2(x)\leq \frac{C}{\text{dist}_{\wti{M}}^2(x,\partial_{rel} U)} \quad \text{for all }\,x\in \Sigma,\]
where $C>0$ is a constant depending only on $C_0$, $U$, and $\partial M\cap U$.
\end{theorem}

The proof of Theorem \ref{thm:main:interior} follows from a similar blow-up argument as in \cite{GLZ16}, but the key observation is a new blow-up scenario inspired by \cite{Zhou-Zhu17}. In particular, by the blow-up argument, if the touching set is still present in the blow-up limit, then by the classical maximum principle for minimal hypersurfaces, the blow-up limit coincides with the tangent plane of the boundary, and hence is flat. This would be a contradiction to the blow-up assumption.

The second main ingredient is a new scheme to construct a non-trivial Jacobi field on a FBMH which is a limit of a sequence of non-identical FBMHs in the case that a sequence of boundary components of these FBMHs collapses to a point. An illustrative example is the sequence ${\Sigma_i}$ of blow-down of half of the Catenoid, given by
\[ \Sigma_i=\{(x, y, z)\in\mathbb{R}^3: z\geq 0,\, i\cdot \sqrt{x^2+y^2}=\cosh(i\cdot z)\}. \]
When $i\to\infty$, it is straightforward to check that $\Sigma_i$\ converges locally uniformly away from the origin to a limit $\Sigma_\infty$ which is the $x-y$ plane, and the boundary components $\partial\Sigma_i$ collapse to the origin. For any such converging scenario, this point belongs to the touching set of the limit $\Sigma_\infty$, and it is known by \cite{Sim87, Sharp17, ACS17} that the height function between $\Sigma_i$ and $\Sigma_\infty$ after normalization will converge to a Jacobi field away from the origin, but it was not known whether the Jacobi field can be extended across the origin (see \cite[Remark 6]{ACS17}). We design a new height estimate making use of the Morse index bound and prove a Harnack type bound for the normalized height function.

Let us illustrate the ideas for one particular case such that a sequence of necks collapse to a point $p$ in the limit touching set, i.e., $p\in \interior(\Sigma_\infty)\cap \partial M$. 
The first observation is that if the boundary component $\partial \Sigma_i$ has radius of order $d_i$ (then $d_i\to 0$), by the touching structure, we know that the height function of $\partial \Sigma_i$ (with respect to $\Sigma_\infty$) has order $-d_i^2$; (here we assume the interior normal of $\partial M$ points to the positive side). As the second observation, if one covers the boundary $\partial \Sigma_i$ by balls of radius $d_i^{3/2}$, then the intersection of one of them with $\partial \Sigma_i$ will be stable by a standard combinatorial argument; thus $\Sigma_i$ has curvature bound of order $d_i^{-3/2}$ therein by Theorem \ref{thm:main:interior}. Then one can check that the height of $\Sigma_i$ in a sub-ball of radius $d_i^{7/4}$ will be positive and of order $d_i^{7/4}$, which will be crucial in our estimate (see Theorem \ref{thm:Jacobi} Claim C). Then we consider the height difference between $\Sigma_i$ and one particular leaf in the minimal foliations (surrounding $\Sigma_\infty$). Using a gradient estimate for minimal graphs (see Appendix),  we prove a Harnack type bound between the negative and positive parts of the height function of $\Sigma_i$ (see Theorem \ref{thm:Jacobi} Claim D). This finally leads to the uniform bound of the height function of $\Sigma_i$ and hence the removable of singularity of the Jacobi field.

\vspace{1em}

The paper is organized as follows. We collect some notations and preliminary results in Section \ref{S:preliminaries}.  Theorem \ref{thm:main:interior} will be proved in Section \ref{S:curvature estimates}. The compactness result, Theorem \ref{thm:main compactness}, will be proved in Section \ref{S:compactness}.  We then present the construction of Jacobi fields, Theorem \ref{thm:Jacobi:1}, in Section \ref{S:existence of Jacobi fields}. 

\begin{acknowledgements}
The authors would like to thank Alessandro Carlotto for several useful comments. The second author would like to thank Weiming Shen for many helpful discussions about harmonic functions. 
\end{acknowledgements}

\section{Preliminaries}
\label{S:preliminaries}
In this section, we collect some basic definitions and preliminary results for free boundary minimal hypersurfaces.

Let $M^{n+1}$ be a smooth compact Riemannian manifold with non-empty boundary $\partial M$. We may assume that $M\hookrightarrow \mb{R}^L$ is isometrically embedded in some Euclidean space. By choosing $L$ large, we assume that $M$ is a compact domain of a closed $(n+1)$-dimensional manifold $\wti{M}$. Let $\mathfrak{X}(\mb{R}^L)$ be the space of smooth vector fields in $\mb{R}^L$.  We define the following notation
 \[\mathfrak{X}(M)=\{X\in \mathfrak{X}(\mathbb{R}^L):\,\, X(p)\in T_pM,\, \forall\,\, p\in M\}.\]


\begin{definition}(Almost proper embeddings; \cite{LZ16}). Let $\Sigma^n$ be a smooth $n$-dimensional manifold with boundary $\partial \Sigma$ (possibly empty). A smooth embedding $\phi: \Sigma \to M$ is said to be an \emph{almost proper embedding} of $\Sigma$ into $M$ if  $\phi(\Sigma)\subset M$ and  $\phi(\partial \Sigma)\subset\partial M$. We  write $\Sigma=\phi(\Sigma)$ and $\partial \Sigma=\phi(\partial \Sigma)$.

We use $\mc{S}(\Sigma)$ to denote the touching set $\mathrm{int}(\Sigma)\cap \partial M$ and $\mc{R}(\Sigma)=\Sigma\setminus\mc{S}(\Sigma)$ to denote the {\em proper subset} of $\Sigma$. If the touching set $\mc{S}(\Sigma)$ is empty, then we say that $\Sigma$ is \emph{properly embedded}; otherwise, we say that $\Sigma$ is \emph{improper}.

\end{definition}

It is easy to see that if $\Sigma\subset M$ is improper and $p\in \mc{S}(\Sigma)$, then $\Sigma$ must touch $\partial M$ tangentially at $p$.

Given an almost properly embedded hypersurface  $(\Sigma, \partial \Sigma)\subset (M,\partial M)$, we define
\[\mathfrak{X}(M,\Sigma)=\{X\in \mathfrak{X}(M):\,\, X(p)\in T_p(\partial M),\,\,\forall\, p \text{ lies in a neighborhood of } \partial \Sigma\}.\]

Let $X$ be a compactly supported vector field in $\mathfrak{X}(M,\Sigma)$. Suppose $\phi_t$ is a one-parameter family of diffeomorphisms generated by the vector field $X$ such that $\phi_t(\Sigma)$ is a family of embedded hypersurfaces in $\wti{M}$.
Then the first variation formula gives that
\begin{equation}\label{equ:first}
\delta\Sigma(X)=\frac{d}{dt}\Big|_{t=0}\area(\phi_t(\Sigma))=\int_\Sigma \text{div}_\Sigma X\,da=-\int_\Sigma \lb{H,X}\,da+\int_{\partial \Sigma} \lb{X,\eta}\,ds,
\end{equation}
where $H$ is the mean curvature vector of $\Sigma$ and $\eta$ is the outward unit co-normal to $\partial \Sigma$.

The first variation formula (\ref{equ:first}) implies that $(\Sigma, \partial \Sigma)\subset (M,\partial M)$ is stationary (i.e., $\delta\Sigma(X)=0$ for any compactly supported $X\in \mathfrak{X}(M,\Sigma)$) if and only if the mean curvature of $\Sigma$ vanishes and $\Sigma$ meets $\partial M$ orthogonally along $\partial \Sigma$. Such hypersurfaces are called \emph{free boundary minimal hypersurfaces}.


\subsection{Stability and the Morse index}

Let $\Sigma^n\subset M^{n+1}$ be an almost properly embedded free boundary minimal hypersurface. The quadratic form of $\Sigma$ associated to the
second variation formula is defined as
\[Q(v,v)=\int_\Sigma \left(|\nabla^\perp v|^2- \Ric_M(v,v)-|A^\Sigma|^2|v|^2 \right)\,da  - \int_{\partial \Sigma} h(v,v)\,ds,
\]
where $v$ is a section of the normal bundle of $\Sigma$, $\Ric_M$ is the Ricci curvature of $M$, $A^\Sigma$ and $h$ are the second fundamental forms of the hypersurfaces $\Sigma$ and $\partial M$, respectively. Note that for any compactly supported vector field $X$ in $\mathfrak{X}(M,\Sigma)$, we have $\delta^2\Sigma(X)=Q(X^\perp,X^\perp)$, where $X^\perp$ denotes the projection of $X$ onto the normal bundle of $\Sigma$.

We now define the \emph{Morse index of $\Sigma$ on the proper subset $\mc{R}(\Sigma)$}.
The Morse index of $\Sigma$ on the proper subset is equal to the maximal dimension of a linear subspace of sections of normal bundle $N\Sigma$ compactly supported in $\mc{R}(\Sigma)$ such that the quadratic form $Q(v,v)$ is negative definite on this subspace.

\begin{definition} An almost properly embedded FBMH $\Sigma^n\subset M$ is said to be \emph{stable away from the touching set} $\mc{S}(\Sigma)$ if the Morse index of $\Sigma$ on the proper subset is 0.
\end{definition}

Next we assume that $\Sigma$ is \emph{two-sided}, i.e., there exists a globally defined unit normal vector field $\mf{n}$ on $\Sigma$. Set
\[C_c^\infty(\mc{R}(\Sigma))=\{f \in C^\infty(\Sigma): \,f \text{ vanishes in a neighborhood of }\,\mc{S}(\Sigma) \}. \]
For any smooth function $f \in C_c^\infty(\mc{R}(\Sigma))$, there is a vector field $X\in \mathfrak{X}(M,\Sigma)$ such that $X=f \mf{n}$ on $\Sigma$, which corresponds to  variations vanishing near the touching set $\mc{S}(\Sigma)$.

It is easy to see that if  $\Sigma\subset M$ is  stable away from the touching set, then
\begin{equation}\label{equ:stable}
-\int_\Sigma f\mathcal{L}f\,da+\int_{\partial \Sigma} \Big( f\frac{\partial f}{\partial \eta}- h(\mf{n},\mf{n})f^2\Big)\,ds\geq 0
\end{equation}
for any $f\in C_c^\infty(\mc{R}(\Sigma))$. Here, we use $\mathcal{L}$ to denote the Jacobi operator of $\Sigma$, that is,
\[\mathcal{L}=\Delta_\Sigma+\Ric_M(\mf{n},\mf{n}) +|A^\Sigma|^2.\]

\begin{remark}
In order to emphasize the difference between the  stability away from the touching set and the usual stability, we will say that a FBMH $\Sigma$ is \emph{globally stable} if (\ref{equ:stable}) holds for any $f\in C^\infty(\Sigma)$.
\end{remark}




When $\Sigma$ is two-sided, the Morse index of $\Sigma$ on the proper subset can also be defined as follows (see also \cite[\S 2.4]{Zhou17}).  Suppose that $\Omega$ is a relative open subset of $\mc{R}(\Sigma)$ with smooth boundary. We use $C_0^\infty(\Omega)$ to denote functions $f\in C^\infty(\Omega)$ such that  $f|_{\partial \Omega\setminus \partial \Sigma}=0$. Then we say that a non-zero function $u\in C_0^\infty(\Omega)$ is an eigenfunction with eigenvalue $\lambda$ on $\Omega$ if $u$ satisfies
\[\left\{ \begin{array}{ll}
\mc{L}u=-\lambda u \quad & \text{on }\, \Omega,\\ \frac{\partial u}{\partial \eta}=h(\mf{n},\mf{n})u \quad  & \text{on }\, \partial \Sigma\cap \partial \Omega.
\end{array}  \right. \]
The Morse index of $\Omega$, $\Index(\Omega)$, is equal to the number of negative eigenvalues (counted with multiplicity) on $\Omega$. It is easy to see that if $\Omega_1$ and $\Omega_2$ are two relative open subsets of $\mc{R}(\Sigma)$ with smooth boundary satisfying $\Omega_1\subset \Omega_2$, then $\Index(\Omega_1)\leq \Index(\Omega_2)$.
Then the Morse index of $\Sigma$ on the proper subset is also equal to the $\Index(\mc{R}(\Sigma))$ defined as follows:
\[\Index(\mc{R}(\Sigma))=
\left\{ \Index(\Omega)\,\Big|
\begin{array}{ll}
 &\text{The Morse index of $\Sigma$ is any relative open subset }\\ &  \text{ of $\mc{R}(\Sigma)$ with smooth boundary}
\end{array}
\right\}.\]



\begin{definition} We say that a function $f\in C^\infty(\Sigma)$ is a Jacobi field of $\Sigma$ is $f$ satisfies
\begin{equation}\label{equ:jacobi}
\left\{
\begin{array}{ll}
\Delta_\Sigma f + (\Ric_M(\mf{n},\mf{n}) +|A^\Sigma|^2)f=0\quad  & \text{on}\, \Sigma,\\
\frac{\partial f}{\partial \eta}=h(\mf{n},\mf{n})f \quad & \text{on}\, \partial \Sigma.
\end{array}
\right.
\end{equation}
\end{definition}
We remark that when $\Sigma$ is one-sided, we can still define the Morse index of $\Sigma$ on the proper subset analogously using this exhaustion method by possibly considering the corresponding double covers.


\begin{remark}For simplicity, we will use $\mc{M}$ to denote the set of almost properly embedded FBMHs in $M$. Given $I\in \mathbb{N}$ and $C>0$, we set \[\mc{M}(I,C)=
\left\{ \Sigma\in \mc{M}\,\Big|
\begin{array}{ll}
 &\text{The Morse index of }\Sigma\text{ on the proper subset }\mc{R}(\Sigma)\\ & \text{is bounded by }I\,\text{and }\area(\Sigma)\leq C

\end{array}
\right\}.\]
\end{remark}

We remark that in the proofs of our results, we often allow a constant $C$ to change from line to line, and the dependence of $C$ should  be clear in the context.

\vspace{1em}
In the following, we will recall  some preliminary and known results for FBMHs which will be needed in our proofs.

\subsection{Existence of local minimal foliations}
\label{SS:local foliations}
Let $\wti{M}^{n+1}$ be a closed manifold, and $\Sigma$ an embedded minimal hypersurface in $\wti{M}$ (without boundary). Given $p\in \Sigma$, one can construct a local minimal foliation around $p$ by White \cite[Appendix]{Wh87}. In particular, denote $B_{r}^\Sigma(p)$ by a geodesic ball of $\Sigma$ centered at $p$ with radius $r>0$. When $\epsilon, \eta>0$ are two sufficiently small numbers, we can foliate a small normal neighborhood $B_{\epsilon}^\Sigma(p)\times [-\eta, \eta]$ (in geodesic normal coordinates provided by $\exp: N(\Sigma)\to \wti{M}$, where $N(\Sigma)$ is the normal bundle of $\Sigma$ near $p$) by minimal graphs $v_t$, such that
\[ \begin{split}
& v_0\equiv 0, \quad v_t(x)=t \text{ for } x\in\partial B_{\epsilon}^\Sigma(p), \text{ and }\\
& v_t>0, \text{ if } t>0, \quad v_t<0, \text{ if } t<0.
\end{split}\]
That is to say $B_{\epsilon}^\Sigma(p)\times [-\eta, \eta]$ is a disjoint union of minimal graphs $\{\Sigma_t=\Graph_{v_t}: t\in[-\eta, \eta]\}$, and $\Sigma_0=\Sigma\cap B_{\epsilon}^\Sigma(p)$.  Moreover, these minimal graphs satisfy uniform Harnack inequalities: there exists some uniform constant $C>1$, such that for $t>0$.
\[\frac{t}{C}=\frac{1}{C}\inf_{x\in \partial B_{\epsilon}^\Sigma(p) } v_t(x) \leq v_t(x)\leq C \sup_{x\in \partial B_{\epsilon}^\Sigma(p) } v_t(x)=Ct. \]
The same also holds true for $t<0$ by simply flipping the sign.

\vspace{1em}
We also recall the existence of local free boundary minimal foliations around a boundary point which is due to Ambrozio-Carlotto-Sharp \cite[Proposition 26]{ACS17}. Let $M^{n+1}$ be a compact manifold with boundary, and $\Sigma^n$  an embedded FBMH in $M$. Given a point $p\in \partial \Sigma$, there exists a constant $\delta>0$ such that the half geodesic ball $B_\delta(p)$ in $M$ can be foliated by free boundary minimal leaves $S_t$ with $t\in [-1/2,1/2]$ and $S_0=\Sigma\cap B_\delta(p)$.  Moreover, each slice satisfies a Harnack type estimate (see \cite[Remark 25]{ACS17}).

\subsection{Convergence and Jacobi fields}\label{SS:Jacobi}
Let $\Sigma_k$ be a sequence of  embedded FBMHs in $M$ converging to an embedded two-sided FBMH $\Sigma$ locally smoothly on $\Sigma\setminus \mc{W}$ with multiplicity $m$, where $\mc{W}\subset \Sigma$ is a finite set of points. We now recall the construction of a non-trivial Jacobi field on $\Sigma$ (see \cite[Theorem 5]{ACS17}).  Let $\mf{n}$ be a global unit normal of $\Sigma$ and $X\in \mathfrak{X}(M,\Sigma)$ be an extension of $\mf{n}$. Suppose that $\phi_t$ is a one-parameter family of diffeomorphisms generated by $X$. For any domain $U\subset \Sigma$ and  small $\delta>0$, $\phi_t$ produces a neighborhood $U_\delta$ of $U$ with thickness $\delta$, i.e., $U_\delta=\{\phi_t(x)\,|\, x\in U, |t|\leq \delta\}$. If $U$ is in the interior of $\Sigma$, then $U_\delta$ is the same as $U\times [-\delta,\delta]$ in the geodesic normal coordinates of $\Sigma$ for $\delta$ small.  Now fix a domain $\Omega\subset \subset \Sigma\setminus \mc{W}$, by the convergence $\Sigma_k\to \Sigma$, we know that for $k$ sufficiently large, $\Sigma_k\cap \Omega_\delta$ can be decomposed as $m$ graphs over $\Omega$ which can be ordered by height
\[u_k^1<u_k^2<\cdots<u_k^m.\]

In this paper, we construct a non-trivial Jacobi filed on $\Sigma$ when $m=1$ and $\Sigma_k\neq \Sigma$ eventually. 

Set $\wti w_k=u_k^1/\Vert u_k^1\Vert_{L^2(\Omega)}$. By the computation in \cite[(6.1)]{ACS17}, $\wti{w}_k$ almost satisfies the Jacobi equation (see also  \cite[Claim 5]{Sharp17} and \cite{Sim87}). Moreover, it can be proved that $\wti{w}_k$ is uniformly bounded in $C^l$ norm for all $l$ on any compact subset of $\Omega$ (see \cite[Claim 1]{ACS17}). Thus,  up to a subsequence, $\wti{w}_k$ converges smoothly to a Jacobi field (see (\ref{equ:jacobi})) on $\Omega$. By taking $\Omega_i$ exhausting $\Sigma\setminus \mc{W}$, we obtain a Jacobi field $w$ on $\Sigma\setminus \mc{W}$. Note that a priori, $w$ might be zero. The next step is to show that $w$  smoothly extends across $\mc{W}$ and the following cases were already obtained.
\begin{itemize}
\item If $p\in \mc{W}$ is in the interior of $\Sigma$ and there is a small neighborhood $B_r(p)$ of $p$ such that $B_r(p)\cap \partial\Sigma_k=\emptyset$ for all $k$ sufficiently large, then the existence of local minimal foliations will imply that $w$ smoothly extends across $p$ (see, e.g., \cite[Claim 6]{Sharp17} and \cite{CM25}).
\item If $p\in \mc{W}$ and $p\in \partial \Sigma$, then the existence of local free boundary minimal foliations will give that $w$ smoothly extends across $p$ (see \cite[Claim 2]{ACS17}).
\end{itemize}

The only case left is that if $p\in \mc{W}$ such that  $p\in \partial M$, $p$ is not in the closure of $\partial \Sigma$, and there is a small neighborhood $B_r(p)$ of $p$ such that $B_r(p)\cap \partial \Sigma_k \neq \emptyset$ for all $k$ large enough. We will deal with this case in our proof (see Section \ref{S:existence of Jacobi fields}).


\subsection{Removable singularity results}
We state the following removable singularity results which will be used in our proofs.
\begin{theorem}\label{thm:remove:int}(\cite{SS81})
Let $M^{n+1}$ be a smooth complete manifold. Assume that $2\leq n\leq 6$ and $p\in M$. Suppose that $\Sigma^n\subset M \setminus\{p\}$ is a smooth embedded minimal hypersurface with $\area(\Sigma)\leq C$ for some constant $C>0$. If $\Sigma$ is stable in a punctured geodesic ball $B_\delta(p)\setminus\{p\}$ and $B_\delta(p)\cap (\bar{\Sigma}\setminus\Sigma)=\{p\}$, then $\Sigma$ has removable singularity at $p$.
\end{theorem}

\begin{theorem}(\cite[Theorem 27]{ACS17})\label{thm:remove:boundary}
Let $M^{n+1}$ be a smooth compact manifold with boundary $\partial M$. Assume that $2\leq n\leq 6$ and $p\in \partial M$. Suppose that $\Sigma\subset M\setminus\{p\}$ is a smooth embedded FBMH in $M$ with $\area(\Sigma)\leq C$ for some constant $C$. If $\Sigma$ is stable in a punctured small neighborhood of $p$ and $p$ is in the closure of $\partial \Sigma$, then $\Sigma\cup \{p\}$ is a smooth embedded FBMH.
\end{theorem}

\subsection{Bumpy metrics theorem}

The following bumpy metric theorem is a slightly different variant of Theorem 9 in \cite{ACS17}, and the proof follows along the same lines as in \cite[Theorem 9]{ACS17}.

\begin{theorem}(\cite[Theorem 9]{ACS17})\label{thm:bumpy}
Let $\wti{M}^{n+1}$ be a smooth closed manifold and $N^n$ be a smooth embedded closed hypersurface in $\wti{M}$.  Suppose that $k$ is an integer $\geq 3$ or that $k=\infty$. Then a generic $C^k$ Riemannian metric on $\wti{M}$ is bumpy in the following sense: if $\Sigma^n$ is an embedded FBMH in $\wti{M}$ with free boundary lying on $N$, then $\Sigma$ or its finite-sheeted covering has no non-trivial Jacobi fields.
\end{theorem}

\begin{remark}
Compared with \cite[Theorem 9]{ACS17}, the FBMH $\Sigma^n$ in Theorem \ref{thm:bumpy} is allowed to penetrate the constraint hypersurface $N$. The arguments of \cite[Theorem 9]{ACS17} still hold since the proof  in \cite{ACS17} identifies a tubular neighborhood of $\Sigma$ with that of the zero section in the normal bundle of $\Sigma$.
\end{remark}

\subsection{Index bound}
We recall the following two lemmas which, roughly speaking, give  ``ball coverings" of unstable regions of FBMHs with finite index (cf. \cite[Lemma 3.1, Lemma 3.2]{LiZh16} and \cite[Lemma 3.1]{Sharp17}), and the proofs are similar to those of \cite[Lemma 3.1]{Sharp17} and \cite[Lemma 3.2]{LiZh16}.
\begin{lemma}
\label{lemma:combinatorial pick up}
Let $M^{n+1}$ be a compact Riemannian manifold with boundary $\partial M$. For fixed $I\in \mathbb{N}$,  suppose $\Sigma\in \mc{M}$ is a smooth FBMH in $M$ such that the Morse index of $\Sigma$ on the proper subset $\mc{R}(\Sigma)$ is bounded by $I$. Given any disjoint collection of $I+1$ open sets $\{U_i\}_{i=1}^{I+1}\subset M$, then $\Sigma$ must be stable away from the touching set in $U_i$ for some $1\leq i\leq N+1$.
\end{lemma}

\begin{lemma}\label{lemma:index}
Let $M^{n+1}$ be a compact Riemannian manifold with boundary $\partial M$. For fixed $I\in \mathbb{N}$,  suppose $\Sigma\in \mc{M}$ is a smooth FBMH in $M$ such that the Morse index of $\Sigma$ on the proper subset $\mc{R}(\Sigma)$ is bounded by $I$. Then for any $r$ small enough, there exist at most $I$ disjoint balls $\{B_r(p_i)\}_{i=1}^I$ of $M$ such that $\Sigma$ is stable away from the touching set on any ball $B_r(x)\text{ in } M\setminus \cup_{i=1}^I B_r(p_i)$.
\end{lemma}

\section{Curvature estimates}
\label{S:curvature estimates}

In this section, we will prove the curvature estimates Theorem \ref{thm:main:interior}. First, let us recall the curvature estimate in \cite{GLZ16}  for FBMHs which are globally stable (across the touching set).
\begin{theorem}(\cite[Theorem 1.1]{GLZ16})
\label{thm:curv:old}
Let $(M^{n+1},g)$ be a compact Riemannian manifold with boundary $\partial M$ and $2\leq n\leq 6$. Suppose that $U \subset M$ is a relative open subset. If $\Sigma^n\subset U$ is an embedded stable minimal hypersurface in $M$ with free boundary lying on $\partial M\cap U$ and $\area(\Sigma)\leq C_0$, then
\[|A|^2(x)\leq \frac{C}{\text{dist}_M^2(x,\partial_{rel} U)} \quad \text{for all }\,x\in \Sigma,\]
where $C>0$ is a constant depending only on $C_0$, $U$, and $\partial M\cap U$.
\end{theorem}

Compared with Theorem \ref{thm:curv:old}, the key novelty of Theorem \ref{thm:main:interior} is that the uniform curvature estimates hold even along the touching set, while we only assume the stability away from it.

In order to prove Theorem \ref{thm:main:interior}, we  use a similar blow-up strategy as in the proof of Theorem \ref{thm:curv:old}. Let $M^{n+1}$ be a compact manifold with boundary $\partial M$. We assume that $M$ is isometrically embedded in some Euclidean space $\mb{R}^L$. Moreover, by choosing $L$ large, we may assume that $M$ is a compact subset of a closed $(n+1)$-dimensional manifold $\wti{M}$.  We use $B_\rho(x)$ to denote the geodesic ball of $\wti{M}$ centered at $x$ with radius $\rho$. We use $\dist(\cdot,\cdot)$ to denote the distance function in $\wti{M}$. Note that the intrinsic distance on $\wti{M}$ and the extrinsic distance on $\mb{R}^L$ are equivalent near any given point, we may assume that the monotonicity formula for FBMHs (\cite[Theorem 3.4]{GLZ16}) holds for geodesic balls with radius less than some $R_0>0$. Theorem \ref{thm:main:interior} will follow directly from the next result.

\begin{theorem}
\label{thm:main:local}
Assume that $2\leq n\leq 6$. Let $M^{n+1}\subset \wti{M}\hookrightarrow \mb{R}^L$, and $R_0$ be given as above. Let $p\in \partial M$ and $0<R<R_0$. Suppose $\Sigma^n\subset B_R(p)$ is a smooth embedded minimal hypersurface in $M$ with free boundary lying on $\partial M\cap B_R(p)$ and $\area(\Sigma)\leq C_0$. If $\Sigma$ is stable away from the touching set $\mc{S}(\Sigma)$, then
\[\sup_{x\in\Sigma \cap B_{R/2}(p)} |A|^2(x)\leq C,\]
where $C$ is a constant depending on $C_0$, $M$ and $\partial M$.
\end{theorem}

\begin{proof}[Proof of Theorem \ref{thm:main:local}]

We need to prove uniform curvature estimate across the touching set $\mc{S}(\Sigma)$. The proof uses a similar strategy as in the proof of Theorem \ref{thm:curv:old}. We will also argue by contradiction.

\vspace{1em}
\textbf{Step 1}: \emph{The blow-up process}.

Suppose the conclusion is false. Then there exists a  sequence of smooth, almost properly embedded, minimal hypersurfaces  $\Sigma_i\subset B_R(p)\cap M$ with free boundary lying on $\partial M\cap B_R(p)$ and  $\area(\Sigma_i)\leq C_0$. Moreover, $\Sigma_i$ is stable away from the touching set $\mc{S}(\Sigma_i)$. But as $i\to \infty$, we have
\[\sup_{x\in \Sigma_i\cap B_{R/2}(p)} |A_i|^2(x)\to \infty.\]
By Theorem \ref{thm:curv:old}, we know that $\mc{S}(\Sigma_i)\cap B_R(p)\neq \emptyset$. We pick up a sequence of points
$x_i\in \Sigma_i \cap B_{R/2}(p)$ such that $|A_i|(x_i)\to \infty$, where $A_i$ denotes the second fundamental form of $\Sigma_i$. 

Since $M$ is compact, there is a subsequence of $x_i$ (still denoted by $x_i$) and a point $x\in M$ so that $x_i\to x$.

We  claim that $x\in \partial M$. If this is not true, then there exists $\rho>0$ such that $B_{3\rho}(x)\cap \partial M=\emptyset$ and $x_i \in B_\rho(x)$ for $i$ sufficiently large.  Note that $\Sigma_i\cap B_\rho(x_i)$ does not intersect $\partial M$. 
Hence, $\Sigma_i\cap B_\rho(x_i)$ is a properly embedded stable minimal hypersurface (with no boundary) in $B_\rho(x_i)$. Moreover, by the classical monotonicity formula and the monotonicity formula for FBMHs, $\area(\Sigma_i\cap B_\rho(x_i))$ is uniformly bounded from above  by $C\rho^n$ for some constant $C$ (depending only on $M$ and the area bound $C_0$). Then by the Schoen-Simon-Yau interior curvature estimate \cite{SSY75} (or Schoen-Simon's curvature estimates \cite{SS81} when $n=6$), we have $\rho^2|A_i|^2(x_i)\leq C_1$,  where $C_1$ is  a uniform constant. This contradicts the assumption that $|A_i|(x_i)\to \infty$. Hence, we conclude that $x\in \partial M$. By similar argument, it follows from  the curvature estimates (Theorem \ref{thm:curv:old}) that $\liminf_{i\to \infty} \dist(x_i,\mc{S}(\Sigma_i))=0$. 

Set
\[ r_i:=\frac{1}{\sqrt{|A_i|(x_i)}}.\] Then we have
\[r_i\to 0,\quad \text{ and }\quad r_i|A_i|(x_i)\to \infty,\quad \text{ as } i\to \infty.\]
Now, we choose $y_i \in \Si_i \cap B(x_i, r_i)$ so that it achieves the maximum of
\begin{equation}
\label{E:y_i-a}
 \sup_{y \in \Si_i \cap B(x_i, r_i)}  |A_i|(y) \dist(y,\partial B(x_i,r_i)) .
\end{equation}
Set $\la_i:=|A_i|(y_i)$ and \[s_i:=r_i-\dist(y_i, x_i)=\dist(y_i,\partial B(x_i,r_i)).\] Since $s_i\leq r_i$, we have $s_i\to 0$ as $i\to \infty$. Using (\ref{E:y_i-a}), we get
\[\begin{array}{rcl}
\la_i s_i & = &|A_i|(y_i)\dist (y_i, \partial B(x_i, r_i))\\      & \geq& |A_i|(x_i)\dist(x_i, \partial B(x_i, r_i)) = r_i |A_i|(x_i).
\end{array}
\]
Hence, we have $\lambda_is_i\to \infty$ as $i\to \infty$. Moreover, the point $y_i \in \Si_i \cap B(x_i, r_i)$ also achieves the maximum of
\begin{equation}
\label{E:y_i-b}
\sup_{y \in \Si_i \cap B(y_i, s_i)}  |A_i|(y) \dist(y,\partial B(y_i,s_i)).
\end{equation}

Let $\eta_i: \R^L \to \R^L$ be the blow-up maps $\eta_i(z):=\la_i (z-y_i)$ centered at $y_i$. Denote $(M'_i, \partial M'_i):=(\eta_i(M),\eta_i(\partial M))$ and let $B'(0,r)$ be the open geodesic ball in $M'_i$ of radius $r>0$ centered at $0 \in M'_i$. We get a blow-up sequence of almost properly embedded  minimal hypersurfaces $ \Si'_i=\eta_i(\Si_i)\subset M_i'$ with free boundary lying on $\partial M'_i$,
which are stable away from their touching sets $\eta_i(\mc{S}(\Sigma_i))$.

Note that
\begin{equation}
\label{E:normalized A_i}
\text{ $|A_i'|(0)=\la_i^{-1}|A_i|(y_i)=1$ for every $i$},
\end{equation}
where $A_i'$ denotes the second fundamental form of $\Sigma_i'$ inside $M'_i$.

For fixed $r>0$, we have $\la_i^{-1} r < s_i$ for all $i$ sufficiently large since $\la_i s_i\to \infty$ as $i\to \infty$. Hence, if $z \in \Si'_i \cap B'(0,r)$, then \[\eta_i^{-1}(z) \in \Si_i \cap B(y_i,\la_i^{-1}r) \subset  \Si_i \cap B(y_i,s_i).\]
This implies that
\[\dist(\eta_i^{-1}(z),\partial B(y_i,s_i)) \geq s_i - \la_i^{-1}r.\]
Combining this with (\ref{E:y_i-b}), we obtain that
\begin{equation}
\label{E:A-x}
|A_i'|(z)\leq\frac{\la_i s_i}{\la_i s_i-r},
\end{equation}
for  $i$ sufficiently large (depending on the fixed $r>0$). Note that the right hand side of (\ref{E:A-x}) approaches $1$ as $i \to \infty$.

\vspace{1em}
\textbf{Step 2}: \emph{The contradiction argument}.

Since $M$ is smooth and $y_i\to x \in \partial M$ as $i\to \infty$, we have that  $B'(0, \la_i s_i)$ converges to $T_x\wti{M}$ smoothly and locally uniformly in $\R^L$. Using the interior curvature estimate for stable minimal hypersurfaces, we have
\[\liminf_{i \to \infty} \la_i \dist_{\R^L}(y_i, \partial M) < \infty.\]
After passing to a subsequence, $\partial M_i'$ converges smoothly and locally uniformly to some $n$-dimensional affine subspace $P\subset T_x\wti{M} \subset \R^L$.

Using the monotonicity formula for FBMHs (see \cite{GLZ16}),  we also have that the blow-ups $\Si_i'$ satisfy a uniform Euclidean area growth with respect to the geodesic balls in $M_i'$ (see Step 3 in the proof of \cite[Theorem 4.1]{GLZ16}).

Since we have the curvature estimate (\ref{E:A-x}) and the uniform Euclidean area growth, the convergence theorem for FBMHs (see \cite[Theorem 6.1]{GLZ16}) implies that there exists a subsequence of $\Sigma_i'$ passing through $0$ which converges smoothly and locally uniformly to either
\begin{itemize}
\item[(a)] a complete embedded minimal hypersurface $\Sigma_\infty^1$ in $T_x\wti{M}$ with Euclidean area growth,  or
\item[(b)] a non-compact, embedded minimal hypersurface $\Sigma_\infty^2$ in $T_x\wti{M}$ with Euclidean area growth and with free boundary on $P$.
\end{itemize}
Due to the smooth convergence and (\ref{E:normalized A_i}), in both cases, we have
\begin{equation}
\label{E:normalized A_infty}
|A_\infty|(0)=1,
\end{equation}
where $A_\infty$ is the second fundamental form of $\Sigma_\infty^1$ or $\Sigma_\infty^2$ in $T_x\wti{M}$.

Note that a priori, both $\Sigma_\infty^1$ and $\Sigma_\infty^2$ may have touching set with $P$. Nevertheless, by the classical maximum principle for minimal hypersurfaces, in Case (a), $\Sigma_\infty^1$ is either identical to $P$, or is disjoint with $P$. However, the first situation cannot happen due to (\ref{E:normalized A_infty}), so $\Sigma_\infty^1$ is disjoint with $P$ and hence is  proper. The locally smooth convergence then implies that $\Sigma_\infty^1$ is globally stable in the classical sense. By similar argument, in Case (b), $\Sigma_\infty^2$ is also  proper, and hence is globally stable with free boundary. For Case (b), since $P$ is a hyperplane, we can double $\Sigma_\infty^2$ by reflecting across $P$ to obtain a complete embedded minimal hypersurface $\wti{\Sigma}_\infty^2$ (see  \cite[Lemma 2.6]{GLZ16}). 

 Since $\Sigma_\infty^1$ or $\wti{\Sigma}_\infty^2$ is embedded with Euclidean area growth and bounded second fundamental form, we conclude that $\Sigma_\infty^1$ or $\wti{\Sigma}_\infty^2$ is properly embedded and, thus, is two-sided.  The classical Bernstein theorem (see \cite{SSY75} and \cite{SS81}) implies that $\Sigma_\infty^1$ or $\wti{\Sigma}_\infty^2$ must be a hyperplane in $T_x\wti{M}$, which contradicts (\ref{E:normalized A_infty}). This completes the proof.
\end{proof}

\begin{remark}
Compared with the proof of Theorem \ref{thm:curv:old}, the key observation in the proof of Theorem \ref{thm:main:local} is that the blow-up limit cannot have touching by the maximum principle.
\end{remark}

\section{Compactness}
\label{S:compactness}

This section is devoted to proving our main compactness result (Theorem \ref{thm:main compactness}). The key part is to prove some removable singularity results. For convenience, we restate Theorem \ref{thm:main compactness} as follows.


\begin{theorem}
\label{thm:cpt}
Assume that $2\leq n\leq 6$. Let $M^{n+1}$ be a compact Riemannian manifold with boundary $\partial M$. For fixed $I\in \mathbb{N}$ and $C_0>0$, suppose that $\{\Sigma_k\}$ is a sequence of FBMHs in $\mc{M}(I,C_0)$. Then there exists a finite set of points $\mathcal{W}\subset M$ with $\#(\mc{W})\leq I$ and a FBMH $\Sigma_\infty\in \mc{M}(I,C_0)$ such that, up to a subsequence, $\Sigma_k$ converges smoothly and locally uniformly to $\Sigma_\infty$ on $\Sigma_\infty\setminus \mc{W}$ with finite multiplicity.
\end{theorem}

\begin{proof}
Let $\{\Sigma_k\}_{k=1}^\infty$ be a sequence in $\mc{M}(I,C_0)$. By compactness of Radon measures, a subsequence, still denoted as $\{\Sigma_k\}$ converges as varifolds to a limit $n$-varifold $V$ which is free stationary (\cite[Definition 2.1]{LZ16}). We are going to prove that $V$ is an integer multiple of some almost properly embedded FBMH.

By Lemma \ref{lemma:index}, for each $\Sigma_k$, there exist at most $I$ disjoint balls $\{B_r(p_{i,k})\}_{i=1}^I$ in $M$ such that $\Sigma_k$ is stable away from the touching set on any ball $B_r(x)\text{ in } M\setminus \cup_{i=1}^I B_r(p_{i,k})$. For fixed $1\leq i\leq I$, after passing to a subsequence, $p_{i,k}$ converges to a point $p_{i,\infty}$. Hence, by possbily shrinking $r$, $\Sigma_k$ is stable away from the touching set on any ball $B_r(p)$ in $ M\setminus \cup_{i=1}^I B_r(p_{i,\infty})$ for $k$ sufficiently large. Using the curvature estimate of Theorem \ref{thm:main:interior}, it follows that there exists a constant $C$ such that for any ball $B_{r/2}(x)$ in $ M\setminus \cup_{i=1}^I B_r(p_{i,\infty})$, we have
\[\sup_{B_{r/2}(x)\cap \Sigma_k} |A|^2\leq \frac{C}{r^2}\]
for $k$ sufficiently large.

Since we have uniform area bound and uniform curvature estimate, a standard compactness argument (see \cite[Theorem 6.1]{GLZ16}) implies that passing to a further subsequence, $\{\Sigma_k\}$ converges to an almost properly embedded FBMH in $ M\setminus \cup_{i=1}^I B_r(p_{i,\infty})$. Letting $r\to 0$, a diagonal argument implies that a further subsequence of $\{\Sigma_k\}$ converges smoothly and locally uniformly to a FBMH $\Sigma$ in $M\setminus\{p_{1,\infty},p_{2,\infty},\ldots,p_{I,\infty}\}$. This also implies that $\area(\Sigma)\leq C_0$. Set
\[\mc{W}=\{p_{1,\infty},p_{2,\infty},\ldots,p_{I,\infty}\}.\]
It is easy to see that $\mc{W}\subset\Clos(\Sigma)$ and the closure $\Clos(\Sigma)$ is identical to the support $\spt(V)$ of $V$. By the monotonicity formula for FBMHs (see \cite[Theorem 3.5]{GLZ16}), a standard argument implies that
\begin{align}\label{equ:Hausdorff}
\text{$\Sigma_k$ converges to $\Clos(\Sigma)$ in Hausdorff distance (even on points in $\mc{W}$)}.
\end{align}

\vspace{1em}
Next, we will show that $\Sigma\cup \mc{W}$ is an almost properly embedded FBMH in $M$. In other words, we will show that each point $p\in \mc{W}$ is a removable singularity of $\Sigma$, or equivalently $\Sigma\cup\{p\}$ is a regular embedded hypersurface in a neighborhood of $p$ in $\widetilde{M}$.

Given any point $p_{i,\infty}\in \mc{W}$, by an analogous argument in \cite[Claim 2]{Sharp17} there exists some $\epsilon_i>0$ such that $\Sigma$ is stable away from the touching set in $B_{\epsilon_i}(p_{i,\infty})\setminus \{p_{i,\infty}\}$.

Let $p$ be a point in $\mc{W}$. We need to consider three cases.

\vspace{1em}

{\em \underline{Case 1:}  Suppose $p$ is in the closure of $\Sigma$ and $p\notin \partial M$}.

 Since there exists some $\epsilon>0$ such that $\Sigma$ is stable away from the touching set in $B_\epsilon(p)\setminus\{p\}$. We can choose $\epsilon$ small so that $B_\epsilon(p)$ is disjoint with the boundary $\partial M$. Hence, $\Sigma$ is globally stable in $B_\epsilon(p)\setminus\{p\}$. It follows directly from the regularity theory of Schoen-Simon \cite{SS81} that $p$ is a removable singularity of $\Sigma$ (see Theorem \ref{thm:remove:int}).

\vspace{1em}
{\em \underline{Case 2:}  Suppose $p\in \partial M$ and $p$ is in the closure of $\partial \Sigma$} (i.e., $p\in \partial M$, ($\partial \Sigma\setminus\{p\})\cap B_\rho(p)\neq \emptyset$ for all $\rho>0$).

Using the boundary removable singularity result, Theorem \ref{thm:remove:boundary},  we see that $p$ is also a removable singularity.

\vspace{1em}
 {\em \underline{Case 3:}  Suppose $p\in \partial M$ and $p$ is not in the closure of $\partial \Sigma$} (i.e., $p\in \partial M$, ($\partial \Sigma\setminus\{p\})\cap B_{\rho_0}(p)=\emptyset$ for some $\rho_0>0$).

If there exists some $\epsilon>0$ such that $\Sigma$ does not touch the boundary $\partial M$ in $B_\epsilon(p)\setminus \{p\}$, then $\Sigma$ is globally stable in $B_\epsilon(p)\setminus \{p\}$, and again the regularity theory of Schoen-Simon \cite{SS81} implies that $p$ is a removable singularity.

If we cannot find such  $\epsilon$, then in any small punctured neighborhood of $p$, $\Sigma$ touches the boundary $\partial M$. Then we consider a blow-up sequence $\lambda_i^{-1}(\Sigma-p)$, where $\lambda_i>0$ is any sequence converging to 0. Using the curvature estimates of Theorem \ref{thm:main:interior} ($\Sigma$ is stable away from the touching set in some punctured  neighborhood of $p$) and the area bound, we obtain that a subsequence of  $\lambda_i^{-1}(\Sigma-p)$ converges smoothly and locally uniformly to a minimal hypersurface $\wti{\Sigma}$ (with possibly integer multiplicity) in $T_p \wti{M} \setminus\{0\}=\mathbb{R}^{n+1}\setminus\{0\}$. Using the classical monotonicity formula, we observe that $\wti{\Sigma}$ must be a cone. Under the rescaling, the subsequence of $\lambda_i^{-1}(\partial M-p)$ converges smoothly to an $n$-dimensional plane $T_p(\partial M)\subset T_p \wti{M}$.

Since $\Sigma$ touches the boundary $\partial M$ in any punctured small neighborhood of $p$, we see that $\wti{\Sigma}$ also touches the plane $T_p(\partial M)$. Note that $T_p(\partial M)$ is also a minimal hypersurface, the classical maximum principle (\cite[Corollary 1.28]{CM4}) implies that $\wti{\Sigma}$ must contains $T_p(\partial M)$, i.e., $T_p(\partial M)\subset \wti{\Sigma}$. Now we argue that $\wti{\Sigma}$ must be identical to $T_p(\partial M)$. If this were not true, there exists another connected components $\wti{\Sigma}_1\subset \wti{\Sigma}$ which is disjoint with $T_p(\partial M)$; moreover $\wti{\Sigma}_1$ is also a cone and $0\in\Clos (\wti{\Sigma}_1)$. By the locally smooth convergence $\lambda_i^{-1}(\Sigma-p)\to \wti{\Sigma}$, we know that $\wti{\Sigma}_1$ must be globally stable, and hence is a plane by Simons' classification (\cite{Sim}). Therefore, by the classical maximum principle again $\wti{\Sigma}_1=T_p(\partial M)$, which is a contradiction, and hence $\wti{\Sigma}=T_p(\partial M)$.

Now we argue that the convergence $\lambda_i^{-1}(\Sigma-p) \to T_p(\partial M)$ must have multiplicity one, and hence by Allard regularity theorem  \cite{Al}, $p$ is a removable singularity of $\Sigma$.  If the convergence $\lambda_i^{-1}(\Sigma-p) \to T_p(\partial M)$ has multiplicity greater than one, then by the locally graphical convergence $\lambda_i^{-1}(\Sigma-p) \to T_p(\partial M)$, near $p$, $\Sigma$ can be decomposed into $m$-graphs over $T_p(\partial M)$ ($m\in\mathbb{N}$), and if denoting the inward unit normal as the $x^{n+1}$-direction, the graphical functions are ordered by height: $u^1<u^2<\cdots< u^m$. Note that by the maximum principle, only the lowest graph, i.e., $\Graph_{u_1}$, can be improper, and all other graphs $\{\Graph_{u_i}: 2\leq i\leq m\}$ must be proper, and hence are regular across $p$ by previous argument. Therefore, the only possibilities are $m=1, 2$. Then we can focus on $\Graph_{u_1}$ and use the above argument to show that it has a unique tangent cone $T_p(\partial M)$ with multiplicity one, and hence $\Graph_{u_1}$ must also extend smoothly across $p$. Therefore, $m=1$ by the maximum principle, and we are done.

This finishes the proof that $\Sigma\cup \mc{W}$ is an almost properly embedded FBMH in $M$. We use $\Sigma_\infty$ to denote $\Sigma\cup \mc{W}$.

\vspace{1em}

Since $\Sigma_k$ converges smoothly to $\Sigma_\infty$ on $\Sigma_\infty\setminus \mc{W}$, a standard argument will give that the Morse index of $\Sigma_\infty$ on the proper subset $\mc{R}(\Sigma_\infty)$ is bounded by $I$ (see \cite[Claim 3]{Sharp17} and \cite{ACS18}). For completeness, we also include a proof here.

We proceed by contradiction. Suppose that the Morse index of $\Sigma_\infty$ on the proper subset $\mc{R}(\Sigma_\infty)$ is greater than $I$. Then there exist $(I+1)$ $L^2$-orthogonal normal vector fields $X_1,\ldots,X_{I+1}$ supported in $\mc{R}(\Sigma_\infty)$ such that the quadratic form of $\Sigma_\infty$ is negative, i.e.,
\[Q(X_i,X_i)<0,\quad \forall\, 1\leq i\leq I+1.\]
We may modify the vector field $X_i$ such that $X_i$ vanishes in a sufficiently small neighborhood of every point in $\mc{W}$ and $Q(X_i,X_i)$ is still negative. In addition,  these vector fields can be extended to be  defined in a small tubular neighborhood of  $\mc{R}(\Sigma_\infty)$ (still denoted by $X_i$). Since $\Sigma_k$ converges smoothly to $\Sigma_\infty$ on $\Sigma_\infty\setminus \mc{W}$, we have
\[\delta^2\Sigma_k(X_i)<0,\]
for $k$ sufficiently large and all $i$.
 Let $X_i^k$   be the projection of $X_i$ onto the normal bundle of $\Sigma_k$.  By the smooth convergence of $\Sigma_k$ to $\Sigma_\infty$ on $\Sigma_\infty\setminus \mc{W}$ and the Hausdorff convergence of $\Sigma_k$ to $\Sigma_\infty$ (see \ref{equ:Hausdorff}), $X_i^k$ has support in $\mc{R}(\Sigma_k)$ for $k$ sufficiently large. Since $\Sigma_k$ has index at most $I$ on $\mc{R}(\Sigma_k)$, we conclude that for fixed $k$, $X_i^k$ ($1\leq i\leq I+1$) are linearly dependent. Without loss of generality, we may assume that
\begin{equation}\label{equ:compt:basis}
\sum_{i=1}^{I}\alpha_i^k X_i^k+X_{I+1}^k=0\quad \text{for some } \{\alpha_i^k\}_{i=1}^{I}\subset \mb{R} \text{ and } |\alpha_i^k|\leq 1.
\end{equation}
Let $Y_i^k$ be the projection of $X_i$ onto the tangent bundle of $\Sigma_k$. The smooth convergence of $\Sigma_k$ to $\Sigma_\infty$ away from the set $\mc{W}$ implies that for every $i$
\[\int_{\Sigma_k} |Y_i^k|^2\to 0, \quad \text{ as }\, k\to \infty.\]
This yields that for
$1\leq i,j\leq I+1$,
\[\lim_{k\to \infty} \int_{\Sigma_k}\lb{X_i^k,X_j^k}=\int_\Sigma \lb{X_i,X_j}=\delta_{ij}. \]
Combining this with the equation (\ref{equ:compt:basis}), we have
\[0=\lim_{k\to \infty} \int_{\Sigma_k} \big|\sum_{i=1}^{I}\alpha_i^k X_i^k+X_{I+1}^k\big|^2=\lim_{k\to \infty} \Big(1+\sum_{i=1}^{I}(\alpha_i^k)^2\Big), \]
which is a contradiction. This completes the proof.
\end{proof}



\section{Existence of Jacobi fields}\label{sec:exist of Jacobi}
\label{S:existence of Jacobi fields}

In this section, we will present the construction of Jacobi fields along the limit of a sequence of FBMHs when the convergence is of multiplicity one, i.e., we will prove Theorem \ref{thm:Jacobi:1}. 
\subsection{Main theorem}

For convenience, we  restate Theorem \ref{thm:Jacobi:1} as follows.
\begin{theorem}\label{thm:Jacobi}
Assume that $2\leq n\leq 6$. Let $M^{n+1}$ be a compact Riemannian manifold with boundary $\partial M$. For fixed $I\in \mb{N}$ and $C_0>0$, suppose that $\{\Sigma_k\}$ is a sequence in $\mc{M}(I,C_0)$, and that $\Sigma_k$ converges smoothly and locally uniformly to some limit $\Sigma\in \mc{M}(I,C_0)$ on $\Sigma\setminus \mc{W}$ with multiplicity one, where $\mc{W}\subset M$ is a finite set of points with $\#(\mc{W})\leq I$.  Assume that $\Sigma_k\neq \Sigma$ eventually. Then
\begin{enumerate}
\item If $\Sigma$ is two-sided, then $\Sigma$ has a non-trivial Jacobi field.

\item If $\Sigma$ is one-sided, then $\wti{\Sigma}$ has a non-trivial Jacobi field, where $\wti{\Sigma}$ is the double cover of $\Sigma$.
\end{enumerate}
\end{theorem}


\begin{proof}[Proof of Theorem \ref{thm:Jacobi}]
We assume that $\Sigma_k\neq \Sigma$ eventually, and proceed to study the nullity of $\Sigma$.

If $\Sigma$ is one-sided, we then consider the orientable double cover $\pi: \wti{\Sigma}\to \Sigma$.  Let $N\Sigma$ be the normal bundle of $\Sigma$. Then the zero section of the pull-back normal bundle $\pi^*(N\Sigma)$ is isometric to $\wti{\Sigma}$, and a small tubular neighborhood of the zero section of $\pi^*(N\Sigma)$ forms a double cover of a tubular neighborhood of $\Sigma$.
We can then construct a non-trivial Jacobi field over $\wti{\Sigma}$ and the construction is similar to the case when $\Sigma$ is two-sided. Hence, in the following, we will assume that $\Sigma$ is two-sided.



\vspace{0.3cm}
Next, we will construct a non-trivial Jacobi field over $\Sigma$.

\begin{figure}[h]
\centering
\includegraphics[height=1in]{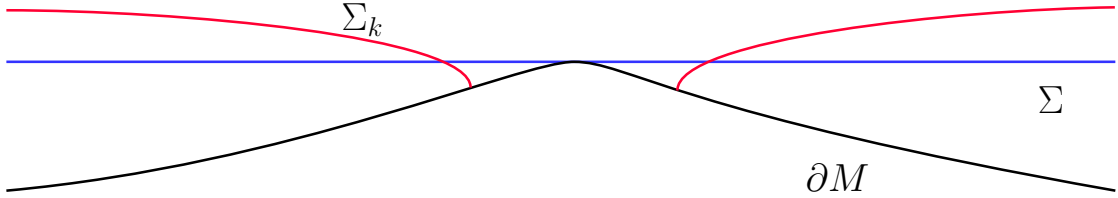}
\caption{Boundary components of $\{\Sigma_k\}$ collapse}
\label{collapse}
\end{figure}

\begin{remark} By the work of \cite[Theorem 5]{ACS17} (see Section \ref{SS:Jacobi}), we only need to deal with the case such that $p\in \mc{W}$, $p \in \partial M$, $p$ is not in the closure of $\partial \Sigma$, and there is a small neighborhood $B_r(p)$ of $p$ such that $B_r(p)\cap \partial \Sigma_k\neq \emptyset$ for $k$ sufficiently large (see Figure \ref{collapse}); the boundary component of $\Sigma_k$ in $B_r(p)$ will be denoted by   $\partial_p\Sigma_k$. 
\end{remark}


Fix $\Omega\subset\subset\Sigma\setminus\mathcal W$. By Section \ref{SS:Jacobi}, for $\delta>0$ small and $k$ sufficiently large, $\Sigma_k\cap \Omega_\delta$ (we refer the readers to Section \ref{SS:Jacobi} for the notion $\Omega_\delta$) can be written as a graph $u_k$ over $\Omega$. Now we set $\wti{u}_k=u_k/ \|u_k\|_{L^2(\Omega)}$. Then we have $\|\wti{u}_k\|_{L^2(\Omega)}=1$. A subsequence of $\wti{u}_k$ will converge smoothly to a Jacobi field (see (\ref{equ:jacobi})) on $\Omega$. Taking an exhaustion $\{\Omega_i\}$ of $\Sigma\setminus \mc{W}$, by a diagonalization argument, we obtain a Jacobi field $u$ on $\Sigma\setminus \mc{W}$. Note that a priori, $u$ may be zero. In the following, we will show that $u$ is non-trivial and can be extended to a global Jacobi field on $\Si$.


Without loss of generality, we can simply assume that $\Sigma_k$ is a normal graph of $u_k$ over $\Omega_k$ (by the assumption of multiplicity one convergence), and that $\mc{W}=\{p\}$ consists of only one point $p$. We first prove that $u$ extends smoothly across $p$, and it suffices to show that $u$ is bounded near $p$ (cf. \cite[Theorem 1.1]{CM25}). Since $\wti{u}_k$ converges to $u$ locally uniformly, we only need to show that, up to a subsequence, $\wti{u}_k$ is uniformly bounded in $\Omega_k$ by taking suitable $\Omega_k$. In particular, fix a small radius $\epsilon>0$, we know that $\wti{u}_k$ converge smoothly and uniformly to $u$ near $\partial B^{\Sigma}_\epsilon(p)$, and the boundary components $\partial_p\Sigma_k$ all lie inside $B^{\wti{M}}_{\epsilon/2}(p)$ for $k$ sufficiently large. The following claim then implies the uniform boundedness of $\wti{u}_k$.

\vspace{0.3cm}

\textit{Claim A: for some constant $C=C(M,\epsilon)>0$,
\[\limsup_{r\rightarrow 0}\limsup_{k\rightarrow\infty}\frac{\max_{B^\Sigma_\epsilon (p)\setminus B^\Sigma_{r}(p)}|\wti u_k|}{\max_{\partial B^\Sigma_\epsilon(p)}|\wti u_k|}\leq C,\]
or equivalently,
\[\limsup_{r\rightarrow 0}\limsup_{k\rightarrow\infty}\frac{\max_{B^\Sigma_\epsilon(p)\setminus B^\Sigma_{r}(p)}|u_k|}{\max_{\partial B^\Sigma_\epsilon(p)}|u_k|}\leq C.\]
 }

To show that $u$ is non-trivial, it sufficies to show that $\|u\|_{L^2(\Si)}>0$. Indeed, by taking $r_i\rightarrow 0,$ and $k_i\rightarrow\infty$ such that
\[ \max_{B^\Sigma_\epsilon(p)\setminus B^\Sigma_{r_i}(p)}|\wti u_{k_i}|\leq C\max_{\partial B^\Sigma_\epsilon(p)}|\wti u_{k_i}|.\]
Setting
\[\Omega_i=\Sigma\setminus B^\Sigma_{r_i}(p),\]
we know that $\|\wti{u}_{k_i}\|_{L^2(B^{\Sigma}_\epsilon(p)\cap \Omega_i)}$ is uniformly small when $\epsilon$ is small. Then the locally uniform convergence of $\wti{u}_{k_i}$ to $u$ away from $p$ and the fact $\|\wti{u}_{k_i}\|_{L^2(\Omega_i)}=1$ imply that $\|u\|_{L^2(\Si\setminus B^{\Sigma}_\epsilon(p))}$ is very close to $1$.

To complete the proof of the theorem, it remains to prove Claim A. The next subsection (\S \ref{pf:Claim A}) is devoted to the proof of Claim A. 
\end{proof}

\subsection{Proof of Claim A}\label{pf:Claim A}
In order to prove Claim A, we need three preliminary results.  In the first result, by using the minimal foliation trick (see Section \ref{SS:local foliations}), we claim that the maximal height $u_k$ in $B^{\Sigma}_\epsilon(p)\times [-\eta, \eta]$ is controlled by that of $u_k$ on the boundary $\partial B^{\Sigma}_\epsilon(p)$.
Here and in the following, we abuse the notation to denote $u_k: \Si_k\to\R$ as the height function in the geodesic normal coordinates of $\Si$. Therefore, by the Hausdorff convergence of $\Sigma_k$ to $\Sigma$ (see (\ref{equ:Hausdorff})), as the height function, $u_k$ is well-defined, and $u_k$ is consistent with the previous definition on the sheet of $\Si_k\cap \Omega_\delta$.


\vspace{0.3cm}
\textit{Claim B: for any $r\in(0,\epsilon)$, there exists some universal constant $C>0$ such that  for $k$ sufficiently large,
\begin{equation}
\label{E:three situations}
\begin{split}
\max_{B^{\Sigma}_r(p)\times [-\eta, \eta]}u_k\leq C\max_{\partial B^{\Sigma}_r(p)}u_k, &\text{ if } \max_{\partial B^{\Sigma}_r(p)}u_k>0,\\
\max_{B^{\Sigma}_r(p)\times [-\eta, \eta]}u_k\leq 0,\quad\quad\quad &\text{ if }  \max_{\partial B^{\Sigma}_r(p)}u_k\leq 0,
\end{split}
\end{equation}
where $B^{\Sigma}_r(p)\times [-\eta, \eta]$ is a geodesic normal neighborhood.}

\begin{proof}[Proof of Claim B]
We want to remark that since $\partial_p\Sigma_k\subset \partial M$ lies underneath $\Si$ (under the assumption that $\mf{n}$ points upward), $u_k|_{\partial_p\Sigma_k}\leq 0$.

First, we assume that  $\max_{\partial B^{\Sigma}_r(p)}u_k>0$. Consider the local minimal foliations $\{\Sigma_t=\Graph_{v_t}: t\in[-\eta, \eta]\}$ of the normal neighborhood $B^{\Sigma}_r(p)\times [-\eta, \eta]$ as described in Section \ref{SS:local foliations}. We can take the leaf $\Sigma_\eta$, which is disjoint with $\Si_k$ for $k$ sufficiently large by the Hausdorff convergence, and decrease the subindex from $\eta$ to $0$. Since the boundary component $\partial_p\Sigma_k$ lies below all $\Si_t$ for $t>0$, by the classical maximum principle for minimal hypersurfaces, $\Sigma_t$ cannot touch $\Si_k$ before $\partial\Sigma_t$ touches $\Si_k\cap (\partial B^{\Sigma}_r(p)\times [-\eta, \eta])$ along the cylinder $\partial B^{\Sigma}_r(p)\times [-\eta, \eta]$. In fact, let $t_1=\max_{\partial B^{\Sigma}_r(p)}u_k>0$ be the first time $\partial\Sigma_t$ touches $\Si_k$ from above (see Figure \ref{t_1}). We then have $\Si_k\cap (B^{\Sigma}_r(p)\times [-\eta, \eta])$ lies under $\Si_{t_1}$, and
\[ \max_{B^{\Sigma}_r(p)\times [-\eta, \eta]}u_k\leq \max_{B^{\Sigma}_r(p)} v_{t_1}\leq C t_1= C\max_{\partial B^{\Sigma}_r(p)}u_k, \]
where we used the Harnark estimates for the minimal foliations in Section \ref{SS:local foliations}.

\begin{figure}[h]
\centering
\includegraphics[height=1.5in]{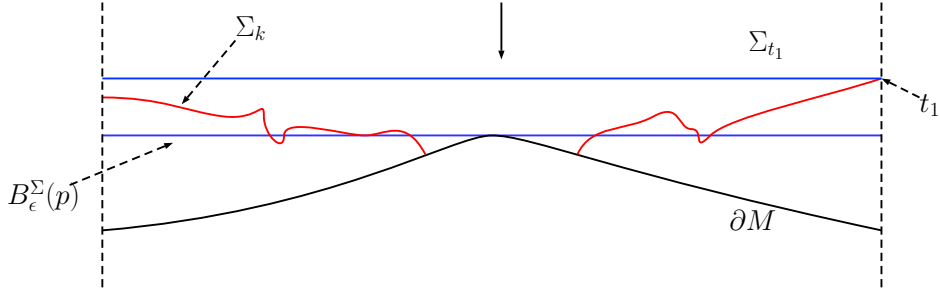}
\caption{Local minimal foliations}
\label{t_1}
\end{figure}

Next, we consider the case $\max_{\partial B^{\Sigma}_r(p)}u_k\leq 0$. By a similar argument as above, we can easily deduce that 
$\max_{B^{\Sigma}_r(p)\times [-\eta, \eta]}u_k\leq 0$.


This finishes the proof of Claim B.
\end{proof}

In the next result, we claim that the for any $r\in (0,\epsilon)$, we always have $\max_{x\in\partial B_r^\Sigma(p)}u_k>0$ for $k$ sufficiently large. This claim essentially uses the assumption that $\Sigma$ has uniform Morse index upper bound.
\vspace{0.3cm}

\textit{Claim C: for any $r\in(0,\epsilon)$, we have  $\max_{x\in\partial B_r^\Sigma(p)}u_k>0$ for $k$ sufficiently large. }
\begin{proof}[Proof of Claim $C$]
Assuming on the contrary that $\max_{\partial B_{2r}^\Sigma(p)}u_k\leq 0$, then $\max_{\partial B_{r}^\Sigma(p)}u_k<0$ by the minimal foliation trick. Let $\pi$ be the projection from $B_r^\Sigma(p)\times[-\eta,\eta]$ to $B_r^\Sigma(p)$. Since $p$ lies in touching set of $\Sigma$, $p\notin \pi(\Sigma_k|_{B_r^\Sigma(p)\times[-\eta,\eta]})$ by our assumptions ($\max_{\partial B_{r}^\Sigma(p)}u_k<0$). Set
\[S_k=B_r^\Sigma(p)\setminus \pi(\Sigma_k|_{B_r^\Sigma(p)\times[-\eta,\eta]}).\]
Then $S_k$ is an open set containing $p$. Denote by $S^0_k$ the connected component of $S_k$ containing $p$. Set
\[T_k=\pi^{-1}(\partial S_k^0)\cap \Sigma_k\cap  (B_r^\Sigma(p)\times[-\eta,\eta]).\]

Then for each $q\in T_k$, either $q\in\partial _p\Sigma_k$ or $\nabla d\in \mathrm{Tan}_q\Sigma_k$, where $d$ is the signed distance function to $\Sigma$. Denote by $2d_k$ the diameter of $S_k^0$.

A key observation is that since $\Sigma$ touches $\partial M$ at $p$, it must touch up to the second order: that is to say, if we write $\partial M$ as a normal graph $v_M$ over $\Si$ near $p$ ($v_M\leq 0$), then
\begin{equation}
\max_{x\in\partial B^{\Sigma}_s(p)}|v_M(x)|\leq C_1 s^2
\end{equation}
for some universal constant $C_1>0$.

To proceed the argument, we consider disjoint collections of geodesic balls (of $\wti{M}$) with centers on $T_k$ and of radius $d_k^{3/2}$. In particular, we can find
\[ \text{a disjoint collection of at least $N\simeq \frac{d_k}{d_k^{3/2}}=d_k^{-1/2}$ many such balls.} \]
Since the Morse index of $\Si_k$ is bounded by $I$, for $r$ small enough (so that $N>I$),  Lemma \ref{lemma:combinatorial pick up} implies that $\Si_k$ is stable away from the touching set in at least one such ball, denoted as $B_{d_k^{3/2}}(q_k)$, where $q_k\in T_k$. Using our curvature estimates Theorem \ref{thm:main:local}, we have
\begin{equation}
\label{E:estimates of A_k}
\sup_{x\in\Sigma_k \cap B_{\frac{1}{2}d_k^{3/2}}(q_k)} |A^{\Si_k}|^2(x)\leq C_2/(d_k^{3/2})^2
\end{equation}
for some uniform constant $C_2>0$.

Denote $\nu$ as the unit inward-pointing normal of $\partial M$; then $\nu(p)=\nabla d(p)$ since $\Sigma$ touches  $\partial M$ at $p$. Now take a geodesic $\gamma_k(t)$ in $\Sigma_k$ satisfying
\begin{itemize}
  \item $\gamma_k(0)=q_k$;
  \item $\gamma_k'(0)=\nu(q_k)$ if $q_k\in \partial_p\Sigma_k$; otherwise, set $\gamma_k'(0)=\nabla d$.
\end{itemize}
Note that we can always take $r$ small enough such that $\lb{\gamma_k'(0),\nabla d}\geq\frac{9}{10}$. Denote by $\mathbf n_k$ the unit normal vector field of $\Sigma_k$. By direct computation,
\begin{align*}
&\frac{d}{ds}\langle \gamma_k'(s),\nabla d\rangle\\
  =&A_k(\gamma'_k(s),\gamma'_k(s))\langle\nabla d,\mathbf n_k\rangle+\nabla^2d(\gamma'_k(s),\gamma'_k(s))\rangle\\
  =&A_k(\gamma'_k(s),\gamma'_k(s))\langle\nabla d,\mathbf n_k\rangle+\nabla^2d((\gamma'_k(s))^T,(\gamma'_k(s))^T),
\end{align*}
where $A_k$ is the second fundamental form of $\Sigma_k$ and $(\gamma'_k(s))^T$ is the projection onto $\{d^{-1}(d(\gamma_k(s)))\}$. Hence for $t\in (0,d_k^{\frac{7}{4}})$,
\[
\frac{d}{dt} d(\gamma_k(t))=  \langle \gamma_k'(t),\nabla d\rangle 
  \geq \frac{9}{10}-\int_0^t (|A_k(\gamma_k(s))|+1)ds
  \geq \frac{9}{10}-Cd_k^{-\frac{3}{2}}t\geq \frac{1}{2}.
\]
Therefore, $\gamma_k(t)\notin \partial_p\Sigma_k$ for $t\in (0,d_k^{\frac{7}{4}})$.

Furthermore, 
\[
  d(\gamma_k(d_k^{\frac{7}{4}}))=d(\gamma_k(0))+\int_0^{d_k^\frac{7}{4}} \lb{\nabla d,\gamma_k'(t)}\, dt\geq -C_1d_k^2+\frac{1}{2}d_k^{\frac{7}{4}}>0,
\]
which leads to a contradiction to our assumptions. This completes the proof of Claim C.
\end{proof}

\textit{Claim D: 
there exists a constant $C=C(M,p,\epsilon)$ such that}
\[\limsup_{r\rightarrow 0}\limsup_{k\rightarrow\infty}\frac{\max_{\partial B_r^\Sigma(p)}|u_k|}{\max_{\partial B_{\epsilon}^\Sigma(p)}u_k}\leq C.\]
\begin{proof}[Proof of Claim D]
Take $r\ll\epsilon$ such that the sectional curvature of $(B_r^\Sigma(p)\times [-r,r],\frac{4}{r^2}g)$ is bounded by 1. For $0<s<t\leq \epsilon$, set
\[\Omega_{s,t}=B_t^\Sigma(p)\setminus B_s^\Sigma(p).\]

By Claim C, we have $\max_{\partial B_r^\Sigma(p)} u_k>0$ for $k$ sufficiently large. Set $t_k=\max_{\partial B_\epsilon^\Sigma(p)}u_k$. Since $u_k(x)<0$ for all $x\in\partial_p\Sigma_k$, $\Sigma_k$ can only touch $\Sigma_{t_k}$ at some points on $\partial B_\epsilon^\Sigma(p)$, where $\Sigma_{t_k}$ is the slice of the minimal foliation on $B_\epsilon^\Sigma(p)$. By Claim B, we know that there exists a constant $C_0$ such that
\[\max_{x\in B_\epsilon^\Sigma(p)}u_k(x)\leq C_0t_k.\]
Moreover, $\Sigma_k\cap(\Omega_{r,\epsilon}\times[-r,r])$ and $\Sigma_{t_k}\cap(\Omega_{r,\epsilon}\times[-r,r])$ converge uniformly smoothly to $\Omega_{r,\epsilon}$ as minimal graphs. Set
\[ \wti B(p)=(B_\epsilon^\Sigma(p), \frac{4}{r^2}g), \quad \wti B_s(q)=(B_{2rs}^\Sigma(q),\frac{4}{r^2}g),\] 
and \[ \wti \Sigma_k=(\Sigma_k,\frac{4}{r^2}g),\quad
  \wti \Sigma_{t}=(\Sigma_t,\frac{4}{r^2}g).
\]

Hence, for any $q\in\partial \widetilde B_4(p)$, $\wti\Sigma_k\cap (\wti B_1(q)\times[-2,2])$ and $\wti\Sigma_{t_k}\cap (\wti B_1(q)\times[-2,2])$ converge uniformly smoothly to $\widetilde B_1(q)$. Then by Corollary \ref{cor:uniform harnack}, there exists $C$ such that
\[\limsup_{k\rightarrow\infty}\frac{\max_{x\in\wti B_{\frac{1}{2}}(q)}\widetilde v_{t_k}-\wti u_k}{\min_{x\in\wti B_{\frac{1}{2}}(q)}\widetilde v_{t_k}-\wti u_k}\leq C,\]
where $\wti v_t$ and $\wti u_k$ are graph functions of $\wti \Sigma_t$ and $\wti \Sigma_k$, respectively.

Now take a finite set $S\subset\partial\wti B_4(p)$ such that
\begin{itemize}
  \item $\partial\wti B_4(p)\subset\cup_{q\in S}\wti{B}_{\frac{1}{2}}(q)$;
  \item $\widetilde B_{\frac{1}{4}}(q_i)\cap \widetilde B_{\frac{1}{4}}(q_j)=\emptyset$ for $i\neq j$.
\end{itemize}

Hence, $\sharp S\leq C'$ for some uniform constant $C'$. By using the Harnack inequality finitely many times, we conclude that for $k$ sufficiently large,
\[\max_{x\in\partial \widetilde B_4(p)}(\widetilde v_{t_k}-\widetilde u_k)\leq C\min_{x\in\partial\widetilde B_4(p)}(\widetilde v_{t_k}-\widetilde u_k).\]
By the definition, $\wti u_k=\frac{2}{r}u_k$ and $\wti v_t=\frac{2}{r}v_t$. Hence, for  $k$  sufficiently large,
\[\max_{x\in\partial B_{2r}^\Sigma(p)}(v_{t_k}-u_k)\leq C\min_{x\in\partial B_{2r}^\Sigma(p)}(v_{t_k}-u_k).\]
Recall that $v_{t_k}>0$ and $\min_{x\in\partial B_{2r}^\Sigma(p)}-u_k=-\max_{x\in\partial B_{2r}^\Sigma(p)}u_k<0$. Assume that 
\[u_k(\overline x)=\max_{x\in\partial B_{2r}^\Sigma(p)}u_k>0.\]
We then have
\[\min_{x\in\partial B_{2r}^\Sigma(p)}(v_{t_k}-u_k)\leq v_{t_k}(\overline x)-u_k(\overline x) \leq C\max_{x\in\partial B_{2r}^\Sigma(p)}v_{t_k}\leq Ct_k=C\max_{x\in\partial B_{\epsilon}^\Sigma(p)}u_k.\]
Therefore,
\[\max_{x\in\partial B_{2r}^\Si (p)}-u_k\leq \max_{x\in\partial B_{2r}^\Si(p)}(v_{t_k}-u_k)\leq C \min_{x\in\partial B_{2r}^\Si(p)}(v_{t_k}-u_k)\leq C\max_{x\in\partial B_{\epsilon}^\Si(p)}u_k.\]
Above all, we have proved that for all $r$ small enough,
\[\limsup_{k\rightarrow\infty}\frac{\max_{\partial B_{2r}^{\Sigma}(p)}|u_k|}{\max_{\partial B_\epsilon^\Sigma(p)}u_k}\leq C,\]
which is the desired inequality. \end{proof}


To proceed the proof of Claim A, we first recall that the argument in Claim B gives the following result:
\[\limsup_{k\rightarrow\infty}\frac{\max_{{B_\epsilon^\Sigma}(p)\setminus B_r^\Sigma(p)}|u_k|}{\max_{\partial B_\epsilon^\Sigma(p)}|u_k|+\max_{\partial B_r^\Sigma(p)}|u_k| }\leq C.\]
Together with Claim D, we finish the proof of Claim A.

\section{Appendix: Harnack inequality for minimal graphs}\label{sec:appendix}

In this appendix, we provide the proof of gradient estimates for minimal graphs which is used in the proof of Claim D in Section \ref{sec:exist of Jacobi}.  Heuristically, it says that the height differences of two sequences of minimal graphs (over a fixed minimal hypersurface) satisfy uniform Harnack estimates if the two sequences converge to that fixed minimal hypersurface.

Let $N\subset M$ be an embedded hypersurface. We always denote by $\mc B^N_r(p)$ the intrinsic geodesic ball of $N$ with radius $r$ and center $p\in N$.
\begin{lemma}\label{lem:app:gradient estimates}
Let $M^{n+1}$ be a closed manifold with $3\leq (n+1)\leq 7$ and $N$ be an embedded compact minimal hypersurface in $M$ so that $\mc B_1^N(p)\cap \partial N=\emptyset$. Suppose that two sequences of embedded compact minimal graphs (over $\mc B^N_1(p)$) $\{\Sigma_k\}$ and $\{\Gamma_k\}$ with graph functions $\{u_k\}$ and $\{v_k\}$ converge smoothly to $\mc B^N_1(p)$ and $u_k-v_k\geq 0$. Then there exists a constant $C=C(M,N)$ such that for any $r>0$ and $p'\in\mc B_{1-r}^N(p)$, we have
\[\limsup_{k\rightarrow\infty}\sup_{x\in\mc B_{r}^N(p')}(r-\dist_N(x,p'))|\nabla\log(u_k-v_k)(x)|\leq C.\]
\end{lemma}
\begin{proof}
We will prove by contradiction. So assume on the contrary that there are two sequences of minimal graphs $\{\Sigma_k\}$ and $\{\Gamma_k\}$ with graph functions $\{u_k\}$ and $\{v_k\}$ over $\mc B_1^{N}(p)$ smoothly converging to $\mc B_1^{N}(p)$, but certain sequences of $r_j>0$ and $p_j\in \mc B_{1-r_j}^N(p)$ satisfying 
\[\limsup_{k\rightarrow\infty}\sup_{x\in \mc B_{r_j}^N(p_j)} (r_j-\dist_N(x,p_j))|\nabla\log (u_k-v_k)(x)|>j,\]
Now we can choose $k(j)$ large enough so that $r_j^{-4}(\Vert u_{k(j)}\Vert_{C^4}+\Vert v_{k(j)}\Vert_{C^4})\rightarrow 0$ and 
\[\sup_{x\in \mc B_{r_j}^N(p_j)} (r_j-\dist_N(x,p_j))|\nabla\log (u_{k(j)}-v_{k(j)})(x)|>j,\]

By choosing $\delta\ll 1$ sufficiently small, $\mc B^N_1(p)\times [-\delta,\delta]$ is a small normal neighborhood (in geodesic normal coordinate) of $\mc B_1^N(p)$ in $M$. Let $\pi$ be the projection map from $\mc B^N_1(p)\times [-\delta,\delta]$ to $\mc B_1^N(p)$. For any $q\in\mc B_1^N(p)\times [-\delta,\delta]$, we use $\mc B_r(q)$ to denote the slice $\mc B^N_r(p)\times \{d(q)\}$ in the normal coordinate, where $d$ is the oriented distance function to $\mc B_1^N(p)$ in $\mc B^N_1(p)\times [-\delta,\delta]$.

Now, we assume that $q_j\in \mc B_{r_j}^N(p_j)$ satisfies
\[(r_j-\dist_N(q_j,p_j))|\nabla \log (u_{k(j)}-v_{k(j)})(q_j)|=\sup_{x\in \mc B_{r_j}^N(p_j)}(r_j-\dist(x,p_j))|\nabla \log (u_{k(j)}-v_{k(j)})(x)|.\]
Set 
\[\rho_j=\frac{r_j-\dist_N(q_j,p_j)}{2},\quad q_j'=\pi^{-1}(q_j) \cap \Sigma_{k(j)},\quad \lambda_j=|\nabla \log (u_{k(j)}-v_{k(j)})(q_j)|. \]

Next, we will consider the rescaling of $\mc B_{\rho_j}^N(q_j)\times [-\delta,\delta]$ around the point $q_j'$ by the scale $\lambda_j$. Since $\lambda_j\rho_j>j/2$, we conclude that 
\[(\mc B_{\rho_j}^N(q_j)\times [-\delta,\delta],\lambda_j^2g,q_j')\to (\mb{R}^{n+1},can,\{0\}),\]
and 
\[(\mc B_{\rho_j}(q_j'),\lambda_j^2g,q_j')\rightarrow(P,can,\{0\}),\]
where $P$ is a hyperplane in $\mb R^{n+1}$.

Let $\wti{\Sigma}_j$ and $\wti{\Gamma}_j$ be the dilations of $\Sigma_{k(j)}$ and $\Gamma_{k(j)}$, respectively.
\begin{claim}\label{claim:app:Sigmak to plane after blowup}
$\wti{\Sigma}_j$ converges to $P$.
\end{claim}
\begin{proof}[Proof of the Claim \ref{claim:app:Sigmak to plane after blowup}]
Denote by $\tilde{u}_j$ and $\tilde{v}_j$ the corresponding graph functions of $\wti{\Sigma}_j$ and $\wti{\Gamma}_j$ over $(\mc B_{r_j}^N(p_j),\lambda_j^2g)$, and set 
\[w_j=\tilde{u}_j-\tilde{v}_j.\]
Note that $\wti{\Sigma}_j$ and $\wti{\Gamma}_j$ are also minimal graphs over $(\mc B_{\rho_j}(q_j'),\lambda_j^2g)$ with graph functions 
\[\bar{u}_j(x)=\tilde{u}_j(\pi(x))-\lambda_j d(q_j')\,\,\text{ and }\,\,\bar{v}_j(x)=\tilde{v}_j(\pi(x))-\lambda_j d(q_j'),\, \text{ respectively}. \] 
Note that $(\mc B_{\rho_j}(q_j'),\lambda_j^2g,q_j')$ converges to a hyperplane $(P,can,0)$. Also, the smooth convergence of $\Sigma_k$ implies that the second fundamental forms of $\Sigma_k$ are uniformly bounded, and hence, the second fundamental forms of $\wti \Sigma_j$ converges to 0 uniformly. Together with $q_j'\in\wti \Sigma_j$, we have that $\wti\Sigma_j$ locally smoothly converges to a hyperplane $P_1$ containing $0$.

Denote by $\wti{\nabla}$ the Levi-Civita connection of $(\mc B_{\rho_j}(q_j'),\lambda_j^2g)$. Since 
\[|\wti{\nabla} \bar{u}_j(x)|_{\lambda_j^2g}=|\wti{\nabla}(\tilde{u}_k(\pi(x))-\lambda_j d(q_j'))|_{\lambda_j^2g}=|\nabla u_{j}(\pi(x))|_{g}\to 0,\quad \text{as } \, j\to \infty,\]
we conclude that $P_1=P$.
 
\end{proof}

 In the following, the norm of a vector is always under the same metric with the connection and we omit the subscript for simplicity.

\begin{claim}\label{claim:app:Gammak to P}
$\wti{\Gamma}_j$ converges to $P$.
\end{claim}
\begin{proof}[Proof of the Claim \ref{claim:app:Gammak to P}]
Note that we have 
\[|\wti{\nabla} \log w_j(q_j)|=\frac{|\wti{\nabla}  w_j(q_j)|}{|w_j(q_j)|}= \frac{|\nabla (u_{k(j)}-v_{k(j)})(q_j)|}{\lambda_j (u_{k(j)}(q_j)-v_{k(j)}(q_j))}=1.\]
Hence, $w_j(q_j)=|\wti{\nabla}  w_j(q_j)|=|\nabla ( u_{k(j)}-v_{k(j)})(q_j) | \to 0$. This implies that the graph function of $\wti{\Gamma}_j$ (over $(\mc B_{\rho_j}(q_j'),\lambda_j^2g)$) satisfies 
\[\bar{v}_j(x)=\tilde{v}_j(\pi(x))-\lambda_j d(q_j')=\tilde{u}_j(\pi(x))-\lambda_j d(q_j')-w_j(q_j) \to 0. \] 
Therefore, $\wti{\Gamma}_j$ converges to a hyperplane $P_2$ containing $0$. By our assumption, $P_2$ lies on one side of $P$. Hence, we obtain that $P_2=P$.
\end{proof}

We also have that $r_j-\dist(x,p_j)\geq\rho_j$ for any $x\in \mc B^N_{\rho_j}(q_j)$.  This implies that 
\[ |\nabla \log (u_{k(j)}(x)-v_{k(j)}(x))|\leq 2|\nabla \log (u_{k(j)}(q_j)-v_{k(j)}(q_j))| ,\]
and, thus, 
\[|\wti{\nabla}\log  w_j(x) |\leq 2|\wti{\nabla}\log  w_j(q_j) |=2\]
on $(\mc B^N_{\rho_j}(q_j),\lambda_j^2g)$. 

Now we set 
\[\bar{w}_j(x)=w_j(\pi(x)),\quad \text{ and }\,\, h_j(x)=\frac{\bar{w}_j(x)}{\bar{w}_j(q_j')}.\]

Then we have that  $|\wti{\nabla} \log h_j(q_j')|=1$  and $h_j$ is locally uniformly bounded. Since $\wti{\Sigma}_j$ and $\wti{\Gamma}_j$ are minimal graphs over $(\mc B_{\rho_j}(q_j'),\lambda_j^2g)$, then the standard computation \cite[page 331 (5)]{Sharp17} gives that there exists $\epsilon_j\rightarrow 0
$ so that for any $\Omega\subset (\mc B_{\rho_j}(q_j'),\lambda_j^2g)$ and $\eta\in C^2_0(\Omega)$,
\[\Big|\int_{\Omega}\wti\nabla\bar w_j\cdot\wti\nabla\eta\Big|\leq \epsilon_j\int_{\Omega}(|\eta|+|\wti\nabla\eta|)(|\bar w_j|+|\wti \nabla\bar w_j|).\]
By the smooth convergence of $(\mc B_{\rho_j}(q_j'),\lambda_j^2g)$, we conclude that $h_j$ converges to a positive harmonic function $h$ on $P$. Hence, $h$ must be a constant. However, this leads to a contradiction to the fact that $|\wti{\nabla}\log h(0)|=1$. This finishes the proof.

\end{proof}

A direct corollary of Lemma \ref{lem:app:gradient estimates} is the following:
\begin{corollary}\label{cor:uniform harnack}
Let $\mc B_1^N(p), u_k, v_k, r_j,p_j$ be the notations as in Lemma \ref{lem:app:gradient estimates}. Then there exists $C=C(M,N)$ so that
\[\limsup_{j\rightarrow\infty}\limsup_{k\rightarrow\infty}\frac{\sup_{x\in\mc  B_{r_j/2}^N(p_j)}(u_k-v_k)}{\inf_{x\in \mc B_{r_j/2}^N(p_j)}(u_k-v_k)}\leq C.\]
\end{corollary}

\bibliography{minimal_stable_touch}
\bibliographystyle{alpha}
\end{document}